\documentclass[a4paper,11pt]{article}
\usepackage{braket,hyperref}
\usepackage{amssymb,mathtools,amsthm,amsmath}
\usepackage{tikz}
\usepackage{xargs}
\usepackage{enumitem}
\usepackage[colorinlistoftodos,prependcaption,textsize=small]{todonotes}
\usepackage{verbatim}

\theoremstyle{plain}
\newtheorem{theorem}{\bf Theorem}
\newtheorem*{theorem*}{Theorem}

\newtheorem{conjecture}[theorem]{\bf Conjecture}
\newtheorem{proposition}[theorem]{\bf Proposition}

\newtheorem{lemma}[theorem]{\bf Lemma}
\theoremstyle{definition}

\newenvironment{remark}[1][Remark.]{\begin{trivlist}
		\item[\hskip \labelsep {\bfseries #1}]}{\end{trivlist}}

\newcommand{\Rea}{{\mathbb R}}

\DeclareMathOperator{\boxx}{box}
\DeclareMathOperator{\rep}{rep}

\newcommand{\lfrac}[2]{\left\lfloor \frac{#1}{#2}\right\rfloor}

\begin{document}
	\title{Representability and boxicity of simplicial complexes}
\author{Alan Lew\footnote{Department of Mathematics, Technion, Haifa 32000, Israel.		
		 e-mail: alan@campus.technion.ac.il .  Supported by ISF grant no. 326/16.}}
	
	\date{}
	\maketitle

\begin{abstract}
    Let $X$ be a simplicial complex on vertex set $V$. We say that $X$ is $d$-representable if it is isomorphic to the nerve of a family of convex sets in $\mathbb{R}^d$. 
    We define the $d$-boxicity of $X$ as the minimal $k$ such that $X$ can be written as the intersection of $k$ $d$-representable simplicial complexes. This generalizes the notion of boxicity of a graph, defined by Roberts.
    
    A missing face of $X$ is a set $\tau\subset V$ such that $\tau\notin X$ but $\sigma\in X$ for any $\sigma\subsetneq \tau$.     We prove that the $d$-boxicity of a simplicial complex on $n$ vertices without missing faces of dimension larger than $d$ is at most $\left\lfloor\frac{1}{d+1}\binom{n}{d}\right\rfloor$. The bound is sharp: the $d$-boxicity of a simplicial complex whose set of missing faces form a Steiner $(d,d+1,n)$-system is exactly $\frac{1}{d+1}\binom{n}{d}$. 
\end{abstract}

\section{Introduction}

Let $\mathcal{F}=\{F_1,\ldots,F_n\}$ be a family of sets. The \emph{intersection graph} of $\mathcal{F}$ is the graph on vertex set $[n]$, whose edges are the pairs $\{i,j\}$ for $1\leq i<j\leq n$ such that $F_i\cap F_j\neq \emptyset$.
A graph $G=(V,E)$ is called an \emph{interval graph} if it is isomorphic to the intersection graph of a family of compact intervals in the real line.


Let $G$ be a graph. The \emph{boxicity} of $G$, denoted by $\boxx(G)$, is the minimal integer $k$ such that $G$ can be written as the intersection of $k$ interval graphs.
Equivalently, $\boxx(G)$ is the minimal $k$ such that $G$ is isomorphic to the intersection graph of a family of axis-parallel boxes in $\Rea^k$.

The notion of boxicity was introduced by Roberts in \cite{roberts1969boxicity}. The following result was first proved by Roberts in \cite{roberts1969boxicity}, and later rediscovered by Witsenhausen in \cite{witsenhausen1980intersections}:

\begin{theorem}[Roberts \cite{roberts1969boxicity}, Witsenhausen \cite{witsenhausen1980intersections}]\label{thm:roberts}
Let $G$ be a graph with $n$ vertices. Then 
\[
        \boxx(G)\leq \lfrac{n}{2}.
\]
Moreover, $\boxx(G)=\frac{n}{2}$ if and only if $G$ is the complete $\frac{n}{2}$-partite graph with sides of size $2$.
\end{theorem}

Let $\mathcal{F}=\{F_1,\ldots,F_n\}$ be a family of sets. The \emph{nerve} of $\mathcal{F}$ is the simplicial complex
\[
    N(\mathcal{F})=\left\{ \sigma\subset [n]: \, \cap_{i\in \sigma} F_i\neq \emptyset\right\}.
\]
Let $X$ be a simplicial complex. We say that $X$ is \emph{$d$-representable} if it is isomorphic to the nerve of a family $\mathcal{C}$ of compact convex sets in $\Rea^d$. We call the family $\mathcal{C}$ a \emph{representation} of $X$ in $\Rea^d$. The \emph{representability} of $X$, denoted by $\rep(X)$, is the minimal $k$ such that $X$ is $k$-representable.

For every $d\geq 1$, we define the \emph{$d$-boxicity} of $X$, denoted by $\boxx_d(X)$, as the minimal $k$ such that $X$ can be written as the intersection of $k$ $d$-representable simplicial complexes.

Let $G=(V,E)$ be a graph. The \emph{clique complex} of $G$, denoted by $X(G)$, is the simplicial complex on vertex set $V$ whose simplices are the cliques in $G$, that is, the sets $U\subset V$ satisfying $\{u,w\}\in E$ for all $u,w\in U$ such that $u\neq w$.

Let $\mathcal{B}=\{B_1,\ldots,B_n\}$ be a family of axis-parallel boxes in $\Rea^k$. It is well known that any $t$ boxes $B_{i_1},\ldots,B_{i_t}$ have a point in common if and only if $B_{i_j}\cap B_{i_r}\neq \emptyset$ for every $1\leq j<r\leq t$. Therefore, the nerve $N(\mathcal{B})$ is exactly the clique complex of the intersection graph of $\mathcal{B}$.
So, for any graph $G$, we have $\boxx(G)=\boxx_1(X(G))$. Thus, we can see the parameters $\boxx_d(X)$ as higher dimensional generalizations of the  boxicity of a graph.

Let $X$ be a simplicial complex.  A missing face of $X$ is a set $\tau\subset V$ such that $\tau\notin X$ but $\sigma\in X$ for any $\sigma\subsetneq \tau$. Let $h(X)$ be the maximal dimension of a missing face of $X$.
Note that a complex $X$ satisfies $h(X)=1$ if and only if it is the clique complex of some graph $G$ (the missing faces of $X(G)$ are the edges of the complement graph of $G$).

In \cite{witsenhausen1980intersections}, Witsenhausen extended Theorem \ref{thm:roberts}, proving that any simplicial complex $X$ with $n$ vertices whose missing faces are all of dimension exactly $d$ has $d$-boxicity at most $\frac{1}{2}\binom{n}{d}$. On the other hand, he showed that a complex $X$ whose missing faces form a Steiner triple system (in particular, $h(X)=2)$ has $2$-boxicity at least $\frac{1}{3}\binom{n}{2}$.

Here, we extend Witsenhausen's lower bound to all values of $d$, and prove an improved upper bound, matching the lower bound:

A family $\mathcal{F}$ of subsets of size $k$ of a set $V$ of size $n$ is called a \emph{Steiner $(t,k,n)$-system} if any subset of $V$ of size $t$ is contained in exactly one set of $\mathcal{F}$. If any subset of $V$ of size $t$ is contained in \emph{at most} one set of $\mathcal{F}$, then $\mathcal{F}$ is called a \emph{partial Steiner $(t,k,n)$-system}.

\begin{theorem}\label{thm:boxd}
Let $X$ be a simplicial complex with $n$ vertices, satisfying $h(X)\leq d$. Then
\[
    \boxx_d(X)\leq \left\lfloor\frac{1}{d+1}\binom{n}{d}\right\rfloor.
\]
Moreover, if $h(X)=d$, then $\boxx_d(X)=\frac{1}{d+1}\binom{n}{d}$ if and only if the missing faces of $X$ form a Steiner $(d,d+1,n)$-system. 
\end{theorem}

Let $\mathbb{F}$ be a field. For $k\geq 0$, let $\tilde{H}_k(X)$ be the $k$-th reduced homology group of $X$ with coefficients in $\mathbb{F}$.
We say that $X$ is \emph{$d$-Leray} if for any induced subcomplex $Y$ of $X$, 
$
    \tilde{H}_k(Y)=0
$  
for all $k\geq d$. The \emph{Leray number} of $X$, denoted by $L(X)$, is the minimal $d$ such that $X$ is $d$-Leray.

It is a well known fact that for any complex $X$,
\[
    L(X)\leq \rep(X).
\]
That is, any $d$-representable complex is $d$-Leray (see e.g. \cite{grunbaum1963helly,Weg75}).


The equality case in Theorem \ref{thm:boxd} follows from the following more general result:

\begin{theorem}\label{prop:lower_bound}
Let $X$ be a complex whose set of missing faces is a partial Steiner $(d,d+1,n)$-system $\mathcal{M}$.
Then, $X$ cannot be written as the intersection of less than $|\mathcal{M}|$ $d$-Leray complexes. 
As a consequence, 
\[
    \boxx_d(X)=  |\mathcal{M}|.
\]
\end{theorem}

It was proved by R{\"o}dl in \cite{rodl1985packing} that, for any $d\geq 1$, there exist partial Steiner $(d,d+1,n)$-systems of size $(1-o(1))\frac{1}{d+1}\binom{n}{d}$. Therefore, the bound in Theorem \ref{thm:boxd} is asymptotically tight. Moreover, by a well known result of Keevash (\cite{keevash2014existence}), there exist Steiner $(d,d+1,n)$-systems for infinitely many values of $n$. Thus, the equality case in Theorem \ref{thm:boxd} is achieved for infinitely many values of $n$.

The upper bound in Theorem \ref{thm:boxd} follows as a consequence of the next result:
\begin{theorem}\label{thm:main_cor_intro_version}
Let $X$ be a simplicial complex on vertex set $V$.  Let $V_1,\ldots, V_k$ be subsets of $V$ satisfying $V_i\notin X$ for all $i\in[k]$, such that for any missing face $\tau$ of $X$ there exists some $i\in[k]$ satisfying $|\tau\setminus V_i|\leq 1$. 

Then, $X$ can be written as an intersection
$
       X= \cap_{i=1}^k X_i,
$
where, for all $i\in[k]$, $X_i$ is a $(|V_i|-1)$-representable complex.
In particular, $X$ is $\left(\sum_{i=1}^k (|V_i|-1)\right)$-representable.
\end{theorem}

The paper is organized as follows. In Section \ref{sec:1}, we present the necessary background on simplicial complexes that we will later need. In Section \ref{sec:intersection} we prove some simple results about the missing faces and the representability of intersections of complexes. Section \ref{sec:lower} contains the proof of Theorem \ref{prop:lower_bound}. In Section \ref{sec:upper} we prove Theorem \ref{thm:main_cor_intro_version}. In Section \ref{sec:boxd} we prove our main result, Theorem \ref{thm:boxd}. In Section \ref{sec:conc} we present some related open problems.

\section{Preliminaries}\label{sec:1}

For any set $U$, the \emph{complete complex} on vertex set $U$ is the complex
\[
    2^U=\{ \sigma: \sigma\subset U\}.
\]  
For $0\leq k\leq |U|-1$, the \emph{complete $k$-dimensional skeleton} on vertex set $U$ is the complex
\[
  \{\sigma\subset U:\, |\sigma|\leq k+1\}.
\]

Let $X$ be a simplicial complex on vertex set $V$. For $U\subset V$, the \emph{subcomplex of $X$ induced by $U$} is the complex
\[
    X[U]= \{ \sigma\in X:\, \sigma\subset U\}.
\]
Let $\mathcal{M}$ be the set of missing faces of $X$. Let
\[
    \Gamma(X)= \left\{\mathcal{N}\subset \mathcal{M}: \, \bigcup_{\tau\in\mathcal{N}}\tau \neq V\right\}.
\]
Note that $\Gamma(X)$ is a simplicial complex on vertex set $\mathcal{M}$. The homology groups of $X$ and $\Gamma(X)$ are related as follows:
\begin{theorem}[Bj{\"o}rner, Butler, Matveev \cite{bjorner1997note}]\label{thm:alexander_nerve}
If $X$ is not the complete complex on $V$, then for all $k\geq 0$,
\[
    \tilde{H}_k(X) \cong \tilde{H}_{|V|-k-3}(\Gamma(X)).
\]
\end{theorem}

Finally, we will need the following simple property of $d$-Leray complexes (see e.g. \cite{grunbaum1963helly}) :
\begin{lemma}\label{lemma:helly_for_leray}
Let $X$ be a $d$-Leray complex. Then,
$
    h(X)\leq d.
$
\end{lemma}

\section{Intersection of simplicial complexes}\label{sec:intersection}

In this section we prove some basic results about the missing faces and the representability of intersections of complexes:

\begin{proposition}
\label{prop:helly_num_of_intersection}
Let $X_1,\ldots,X_k$ be simplicial complexes, and $X=\cap_{i=1}^k X_i$. For each $i\in[k]$, let $\mathcal{M}_i$ be the set of missing faces of $X_i$, and let $\mathcal{M}$ be the set of missing faces of $X$. Then, $\mathcal{M}$ is the set of inclusion minimal elements of $\cup_{i=1}^k \mathcal{M}_i$.
As a consequence, we obtain
\[
    h(X)\leq \max_{i\in[k]} \, h(X_i).
\]
\end{proposition}
\begin{proof}
Let $\tau\in\mathcal{M}$. Since $\tau\notin X$, then there exists some $j\in[k]$ such that $\tau\notin X_j$. Let $\sigma\subsetneq \tau$. Since $\tau$ is a missing face of $X$, we have $\sigma\in X=\cap_{i=1}^k X_i$. In particular, $\sigma\in X_j$. Hence, $\tau$ is a missing face of $X_j$. That is, $\tau\in \mathcal{M}_j \subset\cup_{i=1}^k \mathcal{M}_i$. Moreover, $\tau$ does not contain any other face of $\cup_{i=1}^k\mathcal{M}_i$. Otherwise, there exists some $r\in[k]$ and $\sigma\in\mathcal{M}_r$ such that $\sigma\subsetneq \tau$. Since $\sigma\notin X_r$, then $\sigma\notin X$. But this is a contradiction to $\tau$ being a missing face of $X$.

Now, let $\tau$ be an inclusion minimal element of $\cup_{i=1}^k \mathcal{M}_i$. Then $\tau\in \mathcal{M}_j$ for some $j\in[k]$. In particular, $\tau\notin X_j$, and therefore $\tau\notin X$. Now, let $\sigma\subsetneq \tau$. Assume for contradiction that $\sigma\notin X$. Then, there exists some $r\in[k]$ such that $\sigma\notin X_r$. So, there exists some $\eta\in \mathcal{M}_r$ such that $\eta\subset \sigma\subsetneq \tau$. This is a contradiction to $\tau$ being inclusion minimal in $\cup_{i=1}^k \mathcal{M}_i$. So, $\sigma\in X$. Therefore, $\tau$ is a missing face of $X$.

Since $\mathcal{M}\subset \cup_{i=1}^k \mathcal{M}_i$, we obtain
\[
    h(X)\leq \max_{i\in[k]} \, h(X_i).
\]

\end{proof}

\begin{lemma}\label{lemma:intersection}
Let $X_1,\ldots,X_k$ be simplicial complexes on vertex set $V$. If $X_i$ is $d_i$-representable for each $i\in[k]$, then $\cap_{i=1}^k X_i$ is $\left(\sum_{i=1}^k d_i\right)$-representable.
\end{lemma}
\begin{proof}
For $i\in[k]$, let $\{C_v^i\}_{v\in V}$ be a representation of $X_i$ in $\Rea^{d_i}$. 
For $v\in V$, let
\[
    C_v=C_v^1\times C_v^2 \times \cdots \times C_v^k.
\]
We will show that $\mathcal{C}=\{C_v\}_{v\in V}$ is a representation of $\cap_{i=1}^k X_i$ in $\Rea^{d_1}\times \cdots\times \Rea^{d_k}\cong \Rea^{d_1+\cdots+d_k}$.

Note that the sets $C_v$ are convex, and for any $\sigma\subset V$,
\begin{equation}\label{eq:prod_cross}
    \bigcap_{v\in \sigma} C_v = \left(\bigcap_{v\in \sigma} C_v^1\right) \times \cdots\times \left(\bigcap_{v\in \sigma} C_v^k\right).
\end{equation}

Let $\sigma\subset V$. If $\sigma\in \cap_{i=1}^k X_i$, then $\sigma\in X_i$ for all $i\in[k]$. Hence,  $\cap_{v\in \sigma} C_v^i\neq \emptyset$ for all $i\in[k]$. So, by Equation \eqref{eq:prod_cross},
$
    \cap_{v\in \sigma} C_v\neq \emptyset.
$
If $\sigma\notin \cap_{i=1}^k X_i$, then there exists some $i\in[k]$ such that $\sigma\notin X_i$. Therefore, $\cap_{v\in \sigma} C_v^i=\emptyset$. Thus, by Equation \eqref{eq:prod_cross},
$
    \cap_{v\in \sigma} C_v =\emptyset.
$
Hence, $\mathcal{C}$ is a representation of $\cap_{i=1}^k X_i$ in $\Rea^{d_1+\cdots+d_k}$.
\end{proof}

\section{Lower bounds on $d$-boxicity}\label{sec:lower}

In this section we prove Theorem \ref{prop:lower_bound}. 
For the proof we will need the following simple lemma, which is a generalization of \cite[Lemma 3]{witsenhausen1980intersections}:

\begin{lemma}\label{lemma:obstruction_to_d_leray}
Let $A,B$ be two finite sets, such that $|A|=|B|=d+1$, and $|A\cap B|<d$. Let $V=A\cup B$. Let $X$ be a simplicial complex on vertex set $V$ that has $A$ and $B$ as missing faces, and such that for any other missing face $\tau$ of $X$, $\tau\cup A= V$ and $\tau\cup B= V$. Then, there exists some $k\geq d$ such that $\tilde{H}_k(X)\neq 0$.
\end{lemma}
\begin{proof}
Let $\mathcal{M}$ be the set of missing faces of $X$. Let $\Gamma(X)$ be the simplicial complex
\[
    \Gamma(X)= \left\{\mathcal{N}\subset \mathcal{M}: \, \bigcup_{\tau\in\mathcal{N}}\tau \neq V\right\}.
\]
By assumption, $A\cup B=V$, and for any missing face $\tau\in\mathcal{M}\setminus\{A,B\}$, $A\cup\tau = V$ and $B\cup \tau=V$. Therefore, both $A$ and $B$ are isolated vertices of the complex $\Gamma(X)$. In particular, $\Gamma(X)$ is disconnected. That is, 
\[
    \tilde{H}_0(\Gamma(X))\neq 0.
\]
By Theorem \ref{thm:alexander_nerve}, we have
\[
    \tilde{H}_{|V|-3}(X)= \tilde{H}_0(\Gamma(X))\neq 0.
\]
Since $|A\cap B|<d$, we have
\[
    |V|-3 = |A|+|B|-|A\cap B|-3 \geq 2(d+1)-(d-1)-3 = d.
\]
\end{proof}

\begin{proof}[Proof of Theorem \ref{prop:lower_bound}]
Assume we can write $X$ as
\[
X=\cap_{i=1}^s X_i,
\]
where, for all $i\in [s]$, $X_i$ is a $d$-Leray complex. For each $i\in[s]$, let $\mathcal{M}_i$ be the set of missing faces of $X_i$.

By Proposition \ref{prop:helly_num_of_intersection},
$\mathcal{M}$ is the set of inclusion minimal elements in $\cup_{i=1}^s \mathcal{M}_i$. Since all the elements of $\mathcal{M}$ are of size $d+1$, and all the elements of $\mathcal{M}_i$ are of size at most $d+1$ (since, by Lemma \ref{lemma:helly_for_leray}, the missing faces of a $d$-Leray complex are of dimension at most $d$), 
we must in fact have
\[
    \mathcal{M}= \cup_{i=1}^s \mathcal{M}_i.
\]
(Otherwise, assume there exists some $\tau\in \cup_{i=1}^s  \mathcal{M}_i \setminus \mathcal{M}$. Then, there is some $\eta\in\mathcal{M}$ such that $\eta\subsetneq \tau$. But since all the elements of $\mathcal{M}$ are of size $d+1$, we obtain $|\tau|>d+1$, a contradiction).

Assume for contradiction that $s<|\mathcal{M}|$. Then, by the pigeonhole principle, there exist two distinct sets $\tau_1,\tau_2\in\mathcal{M}$ such that $\tau_1$ and $\tau_2$ are both missing faces of $X_i$ for some $i\in[s]$. Let $\tau_1$ and $\tau_2$ be such a pair with intersection $\tau_1\cap \tau_2$ of maximal size.

Let us look at the induced subcomplex
\[
    Y=X_i[\tau_1\cup\tau_2].
\]
Note that $\tau_1$ and $\tau_2$ are missing faces of $Y$. 
Let $\tau\neq \tau_1,\tau_2$ be a missing face of $Y$. That is, $\tau$ is a missing face of $X_i$ that is contained in $\tau_1\cup\tau_2$.

Let  $k=|\tau_1\cap \tau_2|$, $t=|\tau_1\cap \tau_2 \cap \tau|$, $t_1=|\tau \setminus \tau_2|$ and $t_2=|\tau\setminus \tau_1|$.  Since $\tau\in \mathcal{M}_i\subset \mathcal{M}$, we obtain, by the maximality of $|\tau_1\cap \tau_2|$,
\[
    t_1+t = |\tau\cap \tau_1 | \leq k
\]
and
\[
    t_2+t = |\tau\cap \tau_2 | \leq k.
\]
We obtain
\[
    d+1= |\tau| =t_1+t_2+t \leq 2k-t.
\]
That is,
\[
    t\leq 2k-d-1.
\]
Hence,
\[
    |\tau\setminus(\tau_1\cap \tau_2)|=t_1+t_2 = d+1-t \geq d+1-2k+d+1 = 2(d+1-k).
\]
So, $\tau\setminus(\tau_1\cap \tau_2)$ is a subset of size $t_1+t_2\geq 2(d+1-k)$ of the set $(\tau_1\cup\tau_2)\setminus (\tau_1\cap \tau_2)$. But $|(\tau_1\cup\tau_2)\setminus (\tau_1\cap \tau_2)|=2(d+1-k)$. Therefore, $\tau\setminus (\tau_1\cap \tau_2) = (\tau_1\cup \tau_2)\setminus (\tau_1\cap \tau_2)$.
Hence, we have
\[
    \tau\cup \tau_1 = (\tau\setminus(\tau_1\cap\tau_2))\cup \tau_1 = ((\tau_1\cup \tau_2)\setminus(\tau_1\cap \tau_2))\cup \tau_1 = \tau_1\cup \tau_2,
\]
and similarly
\[
    \tau\cup \tau_2 = (\tau\setminus(\tau_1\cap\tau_2))\cup \tau_2 = ((\tau_1\cup \tau_2)\setminus(\tau_1\cap \tau_2))\cup \tau_2 = \tau_1\cup \tau_2.
\]
Moreover, since $\mathcal{M}$ forms a partial Steiner $(d,d+1,n)$-system, we have $|\tau_1\cap \tau_2|<d$. So, by Lemma \ref{lemma:obstruction_to_d_leray}, $\tilde{H}_r(Y)\neq 0$ for some $r\geq d$. But this is a contradiction to the fact that $X_i$ is $d$-Leray.

Since any $d$-representable complex is $d$-Leray, we obtain:
\[
    \boxx_d(X)\geq |\mathcal{M}|.
\]
On the other hand, it is easy to show that $\boxx_d(X)\leq |\mathcal{M}|$: Let $V$ be the vertex set of $X$. For each $\tau\in\mathcal{M}$, let $X_{\tau}$ be the simplicial complex on vertex set $V$ whose only missing face is $\tau$. It is easy to check that the complex $X_{\tau}$ is $d$-representable (for example, we may assign to each vertex in $\tau$ one of the facets of a simplex $P$ in $\Rea^d$, and assign to all of the vertices in $V\setminus \tau$ the simplex $P$ itself).  Since $X=\cap_{\tau\in\mathcal{M}} X_{\tau}$, we obtain $\boxx_d(X)\leq |\mathcal{M}|$.
\end{proof}

\section{Upper bounds on representability}\label{sec:upper}


In this section we prove Theorem \ref{thm:main_cor_intro_version}.
We will need the following simple lemma:

\begin{lemma}\label{lemma:face_removal}
Let $P\subset \Rea^d$ be a convex polytope. Let $F_1,\ldots,F_m$ be faces of $P$, and let $p_1,\ldots,p_k$ be points in $P$ such that $p_i\notin  F_j$ for all $i\in[k]$ and $j\in[m]$. Then, there exists a convex polytope $P'\subset P$ such that $P'\cap F_j=\emptyset$ for all $j\in[m]$, and $p_i\in P'$ for all $i\in[k]$.
\end{lemma}
\begin{proof}
Let $P'=\text{conv}(\{p_1,\ldots,p_k\})$. Let $j\in[m]$, and let $H$ be a hyperplane supporting $F_j$. That is, $H\cap P=F_j$, and $P$ is contained in one of the closed half-spaces $H^+$ defined by $H$.

Now, since the points $p_1,\ldots,p_k$ belong to $P\setminus F_j$, they must all lie in the interior of $H^+$. Therefore, their convex hull $P'$ is also contained in the interior of $H^+$. Since $F_j$ lies on the boundary $H$ of $H^+$, we have $P'\cap F_j=\emptyset$, as wanted.
\end{proof}

\begin{theorem}
\label{prop:dom_by_one}
Let $X$ be a simplicial complex on vertex set $V$. Let $U\subset V$ such that $U\notin X$ and for any missing face $\tau$ of $X$, $|\tau\setminus U|\leq 1$. Then, $X$ is $(|U|-1)$-representable.
\end{theorem}
\begin{proof}
Let $d=|U|-1$. Let $P$ be a simplex in $\Rea^d$. Assign to each vertex $u\in U$ a facet $F_u$ of $P$. For $\sigma\subset U$, let
\[
    F_{\sigma}=\cap_{u\in\sigma} F_u
\]
(where we understand that $F_{\emptyset}=P$).
Note that, unless $\sigma=U$, $F_{\sigma}$ is a non-empty face of the simplex $P$. For $\sigma\subsetneq U$, let $p_{\sigma}$ be a point in the relative interior of $F_{\sigma}$. Then, for any $\tau\subset U$ and   $\sigma\subsetneq U$, $p_{\sigma} \in F_{\tau}$ if and only if $\tau\subset \sigma$.

Now we build a representation $\{F'_v\}_{v\in V}$ of $X$ in $\Rea^d$, as follows:

We divide into two cases:
\begin{enumerate}

\item Let $u\in U$. Let $\tau\subset U$ and $\sigma\subsetneq U$ such that $u\in \sigma\cap\tau$ and $\tau\notin X$, $\sigma\in X$. Note that $F_{\tau}$ is a face of $F_u$, and $p_{\sigma}\in F_u$. Also, since $X$ is a simplicial complex, we must have $\tau\not\subset \sigma$, and therefore $p_{\sigma}\notin F_{\tau}$. Hence, by Lemma \ref{lemma:face_removal}, there exists a convex polytope $F'_u\subset F_u$ such that $F'_u \cap F_{\tau} =\emptyset$ for all $\tau\subset U$ such that $u\in \tau$ and $\tau\notin X$, and $p_{\sigma}\in F'_u$ for all $\sigma\subsetneq U$ such that $u\in \sigma$ and $\sigma\in X$.

\item Let $v\in V\setminus U$. Let $\tau\subset U$ and $\sigma\subsetneq U$ such that $\tau\cup\{v\}\notin X$ and $\sigma\cup\{v\}\in X$. Since $X$ is a simplicial complex, we must have $\tau\not\subset\sigma$; hence, $p_{\sigma}\notin F_{\tau}$. Therefore, by Lemma \ref{lemma:face_removal}, there exists a convex polytope $F'_v\subset P$ such that $F'_v\cap F_{\tau}=\emptyset$ for all $\tau\subset U$ such that $\tau\cup\{v\}\notin X$ and $p_{\sigma}\in F'_v$ for all $\sigma\subsetneq U$ such that $\sigma\cup\{v\}\in X$.
\end{enumerate}
We will show that the family $\{F'_v\}_{v\in V}$ is a representation of $X$.

First, let $\sigma\in X$. Let $\sigma_1=\sigma\cap U$. Since $\sigma_1\in X$ and $U\notin X$, we have $\sigma_1\subsetneq U$. So, for any $u\in\sigma_1$, we have
\[
    p_{\sigma_1} \in F'_{u}.
\]
Moreover, for any $v\in\sigma\setminus \sigma_1$, since $\sigma_1\cup\{v\}\subset \sigma\in X$, we have
\[
    p_{\sigma_1} \in F'_{v}.
\]
Hence,
\[
 p_{\sigma_1}\in \cap_{v\in \sigma} F'_v.
\]
In particular, $\cap_{v\in\sigma} F'_v\neq \emptyset.$

Now, let $\sigma\subset V$ such that $\sigma\notin X$. Then, there exists some missing face $\tau$ of $X$ such that $\tau\subset \sigma$. By assumption, we have $|\tau\setminus U|\leq 1$.
We divide into two cases:
\begin{enumerate}
    \item Assume $\tau\subset U$. Then, on the one hand, we have
\[
    \cap_{u\in\tau} F'_u \subset \cap_{u\in \tau} F_u =F_{\tau}.
\]
On the other hand, for all $u\in \tau$, by the definition of $F'_u$, we have
\[
    F'_u\cap F_{\tau}=\emptyset.
\]
Hence,
\[
\cap_{u\in\tau} F'_u=\emptyset.
\]
\item Assume that $|\tau\setminus U|=1$. Let $w$ be the unique vertex in $\tau\setminus U$. Then,
\[
 \cap_{u\in \tau\setminus\{w\}} F'_u \subset \cap_{u\in \tau\setminus\{w\}} F_u= F_{\tau\setminus\{w\}}.
\]
But, since $(\tau\setminus \{w\}) \cup \{w\} = \tau\notin X$, we obtain, by the definition of $F'_w$,
\[
    F'_w\cap F_{\tau\setminus\{w\}} = \emptyset.
\]
Hence,
\[
\cap_{v\in \tau} F'_{v} = F'_{w} \cap \left(\cap_{u\in \tau\setminus\{w\}} F'_u\right) 
\subset F'_w \cap F_{\tau\setminus\{w\}} =\emptyset.
\]
\end{enumerate}
In both cases we obtain $\cap_{v\in\tau} F'_v=\emptyset$, and therefore
\[
    \cap_{v\in \sigma} F'_v \subset \cap_{v\in \tau} F'_v =\emptyset.
\]
So, $\{F'_v\}_{v\in V}$ is a representation of $X$ in $\Rea^d=\Rea^{|U|-1}$, as wanted.
\end{proof}

The proof of Theorem \ref{prop:dom_by_one} is based on ideas developed by Wegner in his thesis \cite{wegner1967eigenschaften} (as presented in \cite{eckhoff1993helly,tancer2013intersection}). Indeed, we can think of Theorem \ref{prop:dom_by_one} as an extension of the following result of Wegner:

\begin{theorem}[Wegner \cite{wegner1967eigenschaften}]\label{cor:rep_by_vertices}
Let $X$ be a simplicial complex with $n$ vertices. Then $X$ is $(n-1)$-representable. Moreover, if $X$ is not the complete $(n-2)$-dimensional skeleton, then it is $(n-2)$-representable.
\end{theorem}
\begin{proof}
    If $X$ is the complete complex, then it is trivially $0$-representable. Otherwise, let $U=V$. Since $V\notin X$ and $|\tau\setminus V|=0\leq 1$ for any missing face $\tau$ of $X$, then by Theorem \ref{prop:dom_by_one}, $X$ is $(n-1)$-representable. If $X$ is not the complete $(n-2)$-dimensional skeleton, then there exists some $U\subset V$ of size $n-1$ such that $U\notin X$. Since $|V\setminus U|\leq 1$, then $|\tau\setminus U|\leq 1$ for any missing face $\tau$ of $X$. Hence, by Theorem \ref{prop:dom_by_one}, $X$ is $(n-2)$-representable.
\end{proof}

\begin{proof}[Proof of Theorem \ref{thm:main_cor_intro_version}]
For $i\in[k]$, let $\mathcal{M}_i$ be the set consisting of all the missing faces $\tau$ of $X$ such that $|\tau\setminus V_i|\leq 1$.
Let
\[
    X_i=\{ \sigma\subset V :\, \tau\not\subset \sigma \text{ for all } \tau\in\mathcal{M}_i\}.
\]
Note that $X=\cap_{i=1}^k X_i$. Indeed, if $\sigma\in X$, then $\sigma$ does not contain any missing face of $X$; in particular, for all $i\in[k]$, $\sigma$ does not contain any $\tau\in\mathcal{M}_i$. Therefore, $\sigma\in \cap_{i=1}^k X_i$. On the other hand, if $\sigma\notin X$, then $\tau\subset \sigma$ for some missing face $\tau$ of $X$. By the assumption of the theorem, there exists some $i\in k$ such that $\tau\in\mathcal{M}_i$. So, $\sigma\notin X_i$, and therefore $\sigma\notin \cap_{i=1}^k X_i$.

  Let $i\in[k]$. The set of missing faces of $X_i$ is exactly $\mathcal{M}_i$. Moreover, since $V_i\notin X$, there is some missing face $\tau$ of $X$ such that $\tau\subset V_i$. Since $|\tau\setminus V_i|=0\leq 1$, we have $\tau\in\mathcal{M}_i$; therefore, $V_i\notin X_i$. So, by Theorem \ref{prop:dom_by_one}, $X_i$ is $(|V_i|-1)$-representable. 
  
  Finally, by Lemma \ref{lemma:intersection}, $X$ is $\left(\sum_{i=1}^k (|V_i|-1)\right)$-representable.
\end{proof}

\begin{remark}
In \cite[Theorem 1.2]{ha2014results}, an upper bound similar to the one in Theorem \ref{thm:main_cor_intro_version} is proved for the Leray number of a simplicial complex. Since $L(X)\leq \rep(X)$ for any complex $X$, we can see Theorem \ref{thm:main_cor_intro_version} as a generalization of that result.
\end{remark}

\section{Boxicity of complexes without large missing faces}\label{sec:boxd}

In this section we prove our main result, Theorem \ref{thm:boxd}.

First, we will need the following simple results about Steiner systems:

\begin{lemma}\label{lemma:steinersize}
Let $\mathcal{F}\subset 2^V$ be a partial $(d,d+1,n)$-Steiner system. Then
\[
    |\mathcal{F}|\leq \left\lfloor \frac{1}{d+1}\binom{n}{d}\right\rfloor.
\]
Moreover, if $|\mathcal{F}|=\frac{1}{d+1}\binom{n}{d}$, then $\mathcal{F}$ is a Steiner $(d,d+1,n)$-system.
\end{lemma}
\begin{proof}
Since $\mathcal{F}$ is a partial Steiner $(d,d+1,n)$-system, then any subset of $V$ of size $d$ is contained in at most one element of $\mathcal{F}$. On the other hand, since each $\sigma\in \mathcal{F}$ contains exactly $d+1$ subsets of size $d$, we obtain
\begin{equation}\label{eq:1}
    (d+1)|\mathcal{F}| \leq \binom{n}{d}.
\end{equation}
Therefore, 
\[
    |\mathcal{F}|\leq \left\lfloor \frac{1}{d+1}\binom{n}{d}\right\rfloor.
\]
Now, assume that $|\mathcal{F}|=\frac{1}{d+1}\binom{n}{d}$. Then, equality must hold in \eqref{eq:1}. Thus, each subset of $V$ of size $d$ must be contained in exactly one set of $\mathcal{F}$. That is, $\mathcal{F}$ is a Steiner $(d,d+1,n)$-system.
\end{proof}

\begin{lemma}\label{lemma:claim0}
Let $\mathcal{F}\subset 2^V$ be a $(d,d+1,n)$-Steiner system.
Let $\tau\subset V$ of size $|\tau|\leq d+1$ such that $\tau$ is not contained in any set of $\mathcal{F}$. Then,
\[
    \{ \sigma\in\mathcal{F} : \, |\tau\setminus\sigma|=1 \} \geq d+1.
\]
\end{lemma}
\begin{proof}
Since $\mathcal{F}$ forms a Steiner $(d,d+1,n)$-system, then any set of size at most $d$ is contained in at least one set of $\mathcal{F}$. Therefore, we must have $|\tau|=d+1$. Now, let $\tau_1,\ldots,\tau_{d+1}$ be the subsets of $\tau$ of size $d$. Again, since $\mathcal{F}$ is a Steiner system, there exists $\sigma_1,\ldots,\sigma_{d+1}\in \mathcal{F}$ such that $\tau_i\subset\sigma_i$ for all $i\in[d+1]$.

Since $\tau$ is the only set of size $d+1$ containing two or more of the sets $\tau_1,\ldots,\tau_{d+1}$, but $\tau\notin \mathcal{F}$, we must have  $\sigma_i\neq \sigma_j$ for all $i\neq j$. Thus,
\[
|\{\sigma\in\mathcal{F}:\, |\tau\setminus\sigma|=1\}|\geq |\{\sigma_1,\ldots,\sigma_{d+1}\}|=d+1.
\]
\end{proof}


The last ingredient needed for the proof of Theorem \ref{thm:boxd} is the following result:

\begin{proposition}\label{prop:domination_number}
Let $X$ be a simplicial complex on vertex set $V$ of size $n$, satisfying $h(X)\leq d$. 
Let $t$ be the minimum size of a family $\{\sigma_1,\ldots,\sigma_t\}$ of subsets of size $d+1$ of $V$ satisfying $\sigma_i\notin X$ for all $i\in[t]$, such that for any missing face $\tau$ of $X$, there exists some $i\in[t]$ such that $|\tau\setminus\sigma_i|\leq 1$. 
Then,
\[
    t\leq \left\lfloor\frac{1}{d+1}\binom{n}{d}\right\rfloor.
\]
Moreover, if $h(X)=d\geq 2$, then $t=\frac{1}{d+1}\binom{n}{d}$ if and only if the set of missing faces of $X$ forms a Steiner $(d,d+1,n)$-system.
\end{proposition}
\begin{proof}
Let $\mathcal{M}$ be the collection of all subsets of $V$ of size $d+1$ that are not simplices of $X$.

Let $\mathcal{A}\subset\mathcal{M}$ be a maximal (with respect to inclusion) partial Steiner $(d,d+1,n)$-system. By Lemma \ref{lemma:steinersize}, we have
\[
    |\mathcal{A}|\leq \left\lfloor \frac{1}{d+1}\binom{n}{d}\right\rfloor.
\]
We will show that for any missing face $\tau$ of $X$, there exists some $\sigma\in\mathcal{A}$ such that $|\tau\setminus\sigma|\leq 1$. Assume for contradiction that there exists some missing face $\tau$ of $X$ such that $|\tau\setminus\sigma|>1$ for all $\sigma\in \mathcal{A}$. Let $\sigma_0$ be some set in $\mathcal{M}$ containing $\tau$. Then 
$
    |\sigma_0\setminus \sigma|> |\tau\setminus\sigma|>1
$
for all $\sigma\in\mathcal{A}$. Let $\mathcal{A}'=\mathcal{A}\cup\{\sigma_0\}$. Let $\eta\subset V$ of size $|\eta|=d$. If $\eta\not\subset \sigma_0$, then, since $\mathcal{A}$ is a partial Steiner $(d,d+1,n)$-system, $\eta$ is contained in at most one set in $\mathcal{A}'$. If $\eta\subset \sigma_0$, then assume for contradiction that $\eta\subset \sigma$ for some $\sigma\in\mathcal{A}$. Since $|\sigma_0\setminus \sigma|>1$, we have $|\sigma_0\cap\sigma|\leq d-1$. But this is a contradiction to the fact that $\eta$ is a set of size $d$ contained in $\sigma_0\cap \sigma$. So, $\eta$ is not contained in any set of $\mathcal{A}$. In both cases, $\eta$ is contained in at most one set of $\mathcal{A}'$. Therefore, $\mathcal{A}'\subset\mathcal{M}$ is a partial Steiner $(d,d+1,n)$-system. But this is a contradiction to the maximality of $\mathcal{A}$. 

Therefore, for any missing face $\tau$ of $X$ there exists some $\sigma\in\mathcal{A}$ such that $|\tau\setminus \sigma|\leq 1$.
Hence, \[t\leq |\mathcal{A}|\leq \left\lfloor\frac{1}{d+1}\binom{n}{d}\right\rfloor.\]

Now, assume $t=\frac{1}{d+1}\binom{n}{d}$. Then, we must have  $|\mathcal{A}|=t=\frac{1}{d+1}\binom{n}{d}$. By Lemma \ref{lemma:steinersize}, $\mathcal{A}$ is a Steiner $(d,d+1,n)$-system.

Assume that $h(X)=d\geq 2$. We will show that $\mathcal{A}$ is exactly the set of missing faces of $X$:

 We may assume that $n\geq d+2$. Otherwise, since $h(X)=d$, $X$ must contain a unique missing face of size $d+1$ (that is, $X$ is a Steiner $(d,d+1,d+1)$-system). 

First, we will show that $\mathcal{A}=\mathcal{M}$. Assume for contradiction that there exists some $\tilde{\tau}\in \mathcal{M}\setminus\mathcal{A}$.
By Lemma \ref{lemma:claim0}, there exist $\sigma_1,\sigma_2\in\mathcal{A}$ such that $|\tilde{\tau}\setminus \sigma_1|=|\tilde{\tau}\setminus\sigma_2|=1$. Since $|\tilde{\tau}|=d+1$, we also have $|\sigma_1\setminus\tilde{\tau}|=|\sigma_2\setminus\tilde{\tau}|=1$. Let
\[
    \mathcal{A}'=\mathcal{A}\cup \{\tilde{\tau}\}\setminus\{\sigma_1,\sigma_2\}.
\]

Let $\tau$ be a missing face of $X$. We will show that there exists some $\sigma\in\mathcal{A}'$ such that $|\tau\setminus \sigma|\leq 1$.
We divide into the following cases:
\begin{enumerate}
    \item If $\tau$ is not contained in any set of $\mathcal{A}$, then, by Lemma \ref{lemma:claim0}, we have
\begin{multline*}
    |\{\sigma\in\mathcal{A}':\, |\tau\setminus\sigma|=1\}| \geq |\{\sigma\in\mathcal{A}:\, |\tau\setminus\sigma|=1\}-2 \\
    \geq d+1-2=d-1\geq 1.
\end{multline*}
Therefore, there exists some $\sigma\in\mathcal{A}'$ such that $|\tau\setminus\sigma|=1$.
\item If $\tau$ is contained in some $\sigma\in\mathcal{A}\setminus\{\sigma_1,\sigma_1\}\subset \mathcal{A}'$, then $|\tau\setminus\sigma|=0\leq 1$.
\item If $\tau$ is contained in $\sigma_i$ for some $i\in\{1,2\}$, then
\[
    |\tau\setminus \tilde{\tau}| \leq |\sigma_i\setminus\tilde{\tau}|=1.
\]
\end{enumerate}
Since $|\mathcal{A}'|=t-1$, this is a contradiction to the minimality of $t$. Hence, we must have $\mathcal{A}=\mathcal{M}$.

Finally, assume for contradiction that there exists some missing face $\tau$ of $X$ of size $|\tau|\leq d$. Let $\eta$ be a set of size $d$ containing $\tau$. Then
, since we assumed $n\geq d+2$, we have
\[
 |\{\sigma\subset V:\, |\sigma|=d+1,\,\eta\subset\sigma\}|
 = n-d\geq 2.
\]
Note that any $\sigma\subset V$ such that $|\sigma|=d+1$ and $\eta\subset \sigma$ is not a simplex of $X$ (since it contains the missing face $\tau$), and therefore belongs to $\mathcal{M}=\mathcal{A}$. Hence, $\eta$ is contained in at least two sets of $\mathcal{A}$, a contradiction to $\mathcal{A}$ being a Steiner $(d,d+1,n)$-system. Thus, the set of missing faces of $X$ is exactly $\mathcal{A}$.
\end{proof}

\begin{proof}[Proof of Theorem \ref{thm:boxd}]

Let $\{V_1,\ldots,V_t\}$ be a family of minimum size of subsets of size $d+1$ of $V$ such that $V_i\notin X$ for all $i\in[t]$, and such that for any missing face $\tau$ of $X$, there exists some $i\in[t]$ satisfying $|\tau\setminus V_i|\leq 1$. 
By Theorem \ref{thm:main_cor_intro_version}, we have
$
    \boxx_d(X)\leq t.
$
So, by Proposition \ref{prop:domination_number}, we obtain
\[
    \boxx_d(X)\leq t \leq \left\lfloor\frac{1}{d+1}\binom{n}{d}\right\rfloor.
\]
Now, assume that $h(X)=d$, and the set of missing faces of $X$ does not form a Steiner $(d,d+1,n)$-system. If $d=1$, then it is proved in \cite{witsenhausen1980intersections} that $\boxx_1(X)< \frac{n}{2}$.
If $d\geq 2$ then, by Proposition \ref{prop:domination_number}, we have
\[
    t <\frac{1}{d+1}\binom{n}{d},
\]
and therefore
\[
    \boxx_d(X)\leq t<\frac{1}{d+1}\binom{n}{d}.
\]

Finally, assume that the missing faces of $X$ form a Steiner $(d,d+1,n)$-system $\mathcal{M}$. Then, by Theorem \ref{prop:lower_bound}, we have
\[
    \boxx_d(X)= |\mathcal{M}|= \frac{1}{d+1}\binom{n}{d},
\]
as wanted.
\end{proof}

\begin{remark}
In the case $d=1$, the proof of the upper bound in Theorem \ref{thm:boxd} reduces to the proof of Theorem \ref{thm:roberts} presented by Cozzens and Roberts in \cite{cozzens1983computing}.
\end{remark}

\section{Concluding Remarks}\label{sec:conc}

Let $X$ be a simplicial complex.  By Lemma \ref{lemma:intersection}, we have for any $d\geq 1$,
\[
   \rep(X)\leq d\cdot \boxx_d(X).
\]

In particular, for $d=1$, we obtain as a corollary of Theorem \ref{thm:roberts}:
\begin{proposition}
\label{prop:roberts_rep_version}
Let $G$ be a graph with $n$ vertices. Then, \[
\rep(X(G))\leq \left\lfloor\frac{ n}{2}\right\rfloor.
\]
Moreover, $\rep(X(G))=\frac{n}{2}$ if and only if $G$ is the complete $\frac{n}{2}$-partite graph with all sides of size $2$.
\end{proposition}
The fact that $\rep(X(G))=\frac{n}{2}$ if $G$ is the complete $\frac{n}{2}$-partite graph with sides of size $2$ does not follow directly from Theorem \ref{thm:roberts}. However, it is easy to check that in this case $X(G)$ is the boundary of the $\frac{n}{2}$-dimensional cross-polytope; in particular, it has non-trivial $\left(\frac{n}{2}-1\right)$-dimensional homology group. Thus, $X(G)$ is not $\left(\frac{n}{2}-1\right)$-Leray, and therefore is not $\left(\frac{n}{2}-1\right)$-representable.

We conjecture that for $d\geq 1$, the following extension of Proposition \ref{prop:roberts_rep_version} holds:

\begin{conjecture}
\label{conj:rep_bound}
Let $X$ be simplicial complex with $n$ vertices, satisfying $h(X)\leq d$. Then,
\[
\rep(X)\leq \left\lfloor\frac{d n}{d+1}\right\rfloor.
\]
Moreover, $\rep(X)=\frac{d n}{d+1}$ if and only if the missing faces of $X$ consist of $\frac{n}{d+1}$ pairwise disjoint sets of size $d+1$.
\end{conjecture}

Analogous bounds are known to hold for Leray numbers (see \cite{adamaszek2014extremal}) and for collapsibility numbers (a combinatorial parameter that is bounded from above by the representability of the complex, and bounded from below by its Leray number; see \cite{kim2019complexes}). Conjecture \ref{conj:rep_bound}, if true, would imply both of these results. 

The results presented in this paper do not seem suitable for dealing with Conjecture \ref{conj:rep_bound}. One of the simplest examples where our methods fail is the complex $X_{2,7}$, the complex whose set of missing faces forms a Steiner $(2,3,7)$-system (usually referred to as the Fano plane). Since any two vertices in $X_{2,7}$ are contained in a missing face, the best bound we can obtain from an application of Theorem \ref{prop:dom_by_one} is $\rep(X_{2,7})\leq 5$, which is larger than the conjectured bound $\left\lfloor\frac{2\cdot 7}{3}\right\rfloor=4$. This bound can be proved, however, by the following simple method:

\begin{lemma}\label{lemma:adding_two_simplices}
Let $X$ be a $d$-representable simplicial complex on vertex set $V$. Let $\sigma_1,\sigma_2\subset V$ such that $\sigma_1\cap \sigma_2\in X$.
Then, the complex $X'=X\cup 2^{\sigma_1}\cup 2^{\sigma_2}$ is $(d+1)$-representable.
\end{lemma}
\begin{proof}

Let $e_1,\ldots,e_{d+1}$ be the standard basis for $\Rea^{d+1}$. We identify $\Rea^d$ with the hyperplane $H=\{x\in \Rea^{d+1} : \, x\cdot e_{d+1}=0\}$ in $\Rea^{d+1}$.

Let $\mathcal{P}=\{P_v\}_{v\in V}$ be a representation of $X$ in $\Rea^d$.
Let $x\in \cap_{v\in \sigma_1\cap \sigma_2} P_v\subset H$ (note that $\cap_{v\in \sigma_1\cap \sigma_2} P_v\neq \emptyset$ since $\sigma_1\cap \sigma_2\in X$ and $\mathcal{P}$ is a representation of $X$). Let $x_1=x+e_{d+1}$ and $x_2=x-e_{d+1}$.

For $v\in V$, we define
\[
P_v'=\begin{cases}
\text{conv}(P_v\cup \{x_1\}\cup\{x_2\}) & \text{ if } v\in \sigma_1\cap \sigma_2,\\
\text{conv}(P_v\cup \{x_1\}) & \text{ if } v\in \sigma_1\setminus \sigma_2,\\
\text{conv}(P_v\cup \{x_2\}) & \text{ if } v\in \sigma_2\setminus \sigma_1,\\
P_v & \text{ if } v\notin \sigma_1\cup \sigma_2.
\end{cases}
\]
It is left to the reader to check that $\{P'_v\}_{v\in V}$ is indeed a representation of $X'$.
\end{proof}

\begin{proposition}
\[
    \rep(X_{2,7})\leq 4.
\]
\end{proposition}
\begin{proof}
We identify the vertex set of $X_{2,7}$ with the set $[7]=\{1,2,\ldots,7\}$. Then, the set of missing faces of $X_{2,7}$ is the set
\[
    \mathcal{M}=\{\{1,2,3\},\{1,4,5\},\{1,6,7\},\{2,4,7\},\{3,4,6\},\{2,5,6\},\{3,5,7\}\}.
\]
It is easy to check that the set of maximal faces of $X_{2,7}$ is the set whose elements are the complements of the sets in $\mathcal{M}$:
\begin{align*}
    \{\{4,5,6,7\},\{2,3,6,7\},\{2,3,4,5\},\{1,3,5,6 \}&,
    \\
    \{1,2,5,7\},\{1,3,4,7\},\{1,2,4,6 &\}\}.
\end{align*}
Let $X_0$ be the complex on vertex set $[7]$ whose set of maximal faces is:
\[
\{ \{1,2,4\},\{2,3,4,5\},\{4,5,6,7\}\}.
\]
It can be checked that the following is a representation of $X_0$ in $\Rea^1$:
\begin{align*}
    P_1 &= [0,1], &
    P_2 &= [1,2],\\
    P_3 &= [2,3], &
    P_4 &= [0,5],\\
    P_5 &= [2,5], &
    P_6 &=P_7 = [4,5].
\end{align*}
Let $X_1= X_0 \cup 2^{\{1,2,5,7\}}\cup 2^{\{1,2,4,6\}}$. Since $\{1,2,5,7\}\cap \{1,2,4,6\}=\{1,2\}\in X_0$ then, by Lemma \ref{lemma:adding_two_simplices}, $X_1$ is $2$-representable. 

Let $X_2= X_1\cup 2^{\{1,3\}}\cup 2^{\{2,3,6,7\}}$. Since $\{1,3\}\cap \{2,3,6,7\}=\{3\}\in X_1$ then, by Lemma \ref{lemma:adding_two_simplices}, $X_2$ is $3$-representable.

Finally, let $X_3= X_2 \cup 2^{\{1,3,5,6\}}\cup 2^{\{1,3,4,7\}}$. Since $\{1,3,5,6\}\cap \{1,3,4,7\}= \{1,3\}\in X_2$ then, by Lemma \ref{lemma:adding_two_simplices}, $X_3$ is $4$-representable. But it is easy to check that $X_3$ is in fact the complex $X_{2,7}$.
\end{proof}

Lemma \ref{lemma:adding_two_simplices} gives non-trivial bounds only for complexes with a small number of maximal faces, so it seems unlikely that such a method will be useful in more general cases of the problem.

We conclude with the following problem, whose solution may be a (very modest) step towards Conjecture \ref{conj:rep_bound}:
\begin{conjecture}
    Let $X_{2,9}$ be the simplicial complex whose missing faces form a Steiner $(2,3,9)$-system (that is, they are the lines of the affine plane of order 3). Then,
    \[  
    \rep(X_{2,9})\leq 5.
    \]
\end{conjecture}

\section*{Acknowledgment}
I thank Roy Meshulam for his comments.

\bibliographystyle{abbrv}
\bibliography{biblio}

\end{document}